\documentclass[11pt,reqno]{amsart}
\usepackage{amsfonts}
\usepackage{amssymb}
\usepackage{amsthm}
\usepackage{amsmath}
\usepackage{a4wide}
\usepackage{mathrsfs}
\usepackage{epsfig}
\usepackage{palatino}
\usepackage{tikz,fp,ifthen}
\usepackage{float}
\usepackage{graphicx}
\usepackage{enumitem}
\usepackage{esint}
\usepackage{accents}
\usepackage[colorinlistoftodos]{todonotes}
\usepackage{mathtools}
\usepackage{cite}
\mathtoolsset{showonlyrefs}

\usepackage{color}
\definecolor{citegreen}{rgb}{0,0.6,0}
\definecolor{refred}{rgb}{0.8,0,0}
\usepackage[colorlinks,urlcolor=blue,citecolor=citegreen,linkcolor=refred]{hyperref}
\usepackage{appendix}

\theoremstyle{plain}
\newtheorem{thm}{Theorem}[section]

\newtheorem{proposition}[thm]{Proposition}

\newtheorem{theorem}[thm]{Theorem}

\newtheorem{ackn}{Acknowledgement\!\!}

\theoremstyle{remark}

\newtheorem{remark}[thm]{Remark}

\theoremstyle{definition}
\newtheorem{definition}[thm]{Definition}

\numberwithin{equation}{section}

\newcommand{\R}{\mathbb{R}}

\def\Z{\mathbb Z}
\newcommand{\SSS}{\mathbb S}

\newcommand{\RRR}{{\mathrm R}}
\newcommand{\Ric}{{\mathrm {Ric}}}

\DeclarePairedDelimiter{\abs}{\lvert}{\rvert}

\newcommand{\HHH}{\mathrm{H}} 

\newcommand{\na}{\nabla}

\newcommand{\medint}{-\kern -,375cm\int}
\newcommand{\medintinrigo}{-\kern -,315cm\int}

\newcommand{\beq}{\begin{equation}}
\newcommand{\eeq}{\end{equation}}

\def\ringg#1{\accentset{\circ}{#1}}

\begin{document}
\title[Area, Volume and Capacity in $3$--manifolds with non--negative scalar curvature]{Area, Volume and Capacity in non--compact $3$--manifolds with non--negative scalar curvature}

\author[F.~Oronzio]{Francesca Oronzio}
\address{F.~Oronzio,  Scuola Superiore Meridionale, Largo S. Marcellino, 10, 80138, Napoli, Italy.}
\email{f.oronzio@ssmeridionale.it}

\begin{abstract}
Let $(M,g)$ be a $3$--dimensional, complete, one--ended Riemannian manifold, with a minimal, compact and connected boundary. We assume that $M$ has a simple topology and that the scalar curvature of $(M,g)$ is non--negative. 
Moreover, we suppose that $(M,g)$ admits a $2$--capacitary potential $v$ with $v,\,\vert \nabla v\vert\to 0$ at infinity.
In this note, we provide a gradient integral estimate for the level sets of the function $u=1-v$.
This estimate leads to a sharp volume comparison for the sub--level sets of $u$, and a sharp area comparison of the level sets of $u$. From this last comparison it follows a sharp area--capacity inequality,
originally derived by Bray and Miao in \cite{bray2}, thereby extending its cases of validity.
This work is based on the recent paper \cite{Col_Min_5} by Colding and Minicozzi. Finally, for completeness, we also show the same type of area and volume comparison, in the case where $(M,g)$ has no boundary, replacing the function $u$ with one related to the minimal positive Green's function. This volume comparison leads to a more geometric proof of the positive mass inequality than the one given in \cite{Ago_Maz_Oro_2}.
\end{abstract}

\maketitle


\section{Introduction}

Recently, connections between curvature and the level sets of harmonic and, more generally, $p$--harmonic functions have been discovered. A significant early example of these connections is Colding’s work \cite{Colding_Acta}, where new monotonicity formulas along the level sets of harmonic functions are derived from gradient estimates of the Green’s function. 
A notable implication of this work is the uniqueness of tangent cones in Einstein manifolds, proven by Colding--Minicozzi in \cite{Col_Min_4}.
Subsequently, the same monotone quantities were used in~\cite{Ago_Fog_Maz_1} to obtain new Willmore--type geometrical inequalities and generalized in~\cite{BMF} to prove an optimal version of the Minkowski inequality in nonparabolic Riemannian manifolds with non--negative Ricci
curvature and Euclidean volume growth (see also ~\cite{Ago_Maz_CV, Ago_Fog_Maz_2} for the Euclidean
setting). One of these last monotone quantities is also relevant in a simplification of the proof of the second part of Hamilton's pinching conjecture \cite{Ben_Mant_Oro_Pl, Ben_Leo_Oro_Pl}.

In the same spirit, new monotonicity formulas (even for
non--harmonic functions) were then found in order to study static and sub--static
metrics in general relativity~\cite{Virginia1, Bor_Chr_Maz, Bor_Maz_1,Bor_Maz_2, AMO}.

In the past several years, the level sets of harmonic and $p$-harmonic functions have been used to analyze $3$--manifolds with non--negative scalar curvature. A significant step in this direction was done by Stern in~\cite{Stern}, where an integral inequality linking the scalar curvature of a closed $3$--manifold to harmonic maps into $\SSS^1$ was derived using the Bochner formula, the traced Gauss equation, and the Gauss--Bonnet theorem.
Since then, harmonic and $p$--harmonic functions have found a number of applications in the study of asymptotically flat manifolds with non--negative scalar curvature. Indeed, new proofs of the positive mass theorem were obtained in \cite{Ago_Maz_Oro_2} and \cite{bray3} by means of harmonic functions of different types.
Simpler proofs of the Riemannian Penrose inequality (without rigidity) were found in \cite{AMMO, Hir_Miao_Tam} by exploiting $p$--harmonic functions. 
New results involving the connections ADM--mass/$p$--capacity and area/$p$--capacity were established in \cite{Miao2022, Or_1} via harmonic functions and in\cite{XiaYinZhou, MazYao_1} by means of $p$--harmonic functions.

Outside the asymptotically flat setting, gradient integral estimates along the level sets of the Green's function are a powerful tool to examine $3$--manifolds with non--negative scalar curvature.
These estimates lead to a sharp area comparison for the level sets of the Green's function. 
This comparison was first proved by Munteanu--Wang in \cite{MuntWang}, under the additional assumption of asymptotically non--negative Ricci curvature and later by Chan--Chu--Lee--Tsang \cite{ChChLeTs}, assuming that the Ricci curvature is bounded from below. Very recently, Colding--Minicozzi \cite{Col_Min_5} proved the gradient integral estimates necessary to obtain the sharp area comparison, without the assumption that the Ricci curvature is bounded from below. They used these gradient estimates to show a Cohn--Vessen type bound for the average of the scalar curvature. 
Integral gradient estimates are also a crucial element of the work~\cite{ChodLi} of Chodosh--Li, on the stable minimal $3$--manifolds and the Bernstein problem.

Due to the closeness of some results of this note to those in \cite{MuntWang, ChChLeTs, Col_Min_5}, we want to analyze some common points and differences between them.
All these papers assume the existence of the minimal positive Green's function with a pole vanishing at infinity. Now, when the $3$--manifold has Ricci curvature bounded from below, the Cheng--Yau gradient estimate \cite{Ch_Yau} implies the vanishing of the norm of its gradient at infinity. This property plays a key role in the work of Colding--Minicozzi. Therefore, they remove the assumption on Ricci curvature, but they need to suppose that the norm of the gradient of the Green's function vanishes at infinity. In all three papers it is essential to control the topology of the regular level sets of Green's function, in order to apply the Gauss--Bonnet theorem. This control is obtained by imposing certain conditions on the first integral homology group via the Betti number or by requiring the second integral homology group to be zero.

The aim of this note is to adapt the work of Colding--Minicozzi \cite{Col_Min_5} to $3$--manifolds with a compact, connected, minimal boundary and to present some applications. Among them, the most relevant is the sharp area--capacity inequality, originally derived by Bray--Miao in \cite{bray2}, thus extending its validity cases.
More precisely, we will prove the following theorem.

\begin{theorem}\label{mainTh}
Let $(M,g)$ be a $3$--dimensional, complete, non--compact, one--ended Riemannian manifold, with a minimal, compact and connected boundary. 
Assume that $H_2(M,\partial M;\Z)=0$ and the scalar curvature of $(M,g)$ is non--negative. We suppose that there exists the smooth solution $u$ of problem 
$$\Delta u=0\ \,\, \mathrm{in} \ {M}\quad\quad\quad u=0\ \,\,\mathrm{on} \  \partial M \quad\quad\quad u \to 1\,\, \mathrm{at} \  \infty\,, $$
and that $|\nabla u|$ vanishes at infinity.
Let $\mathcal{C}>0$ be the boundary capacity of $\partial M$ in $(M,g)$, expressed in terms of $u$ as
$$\mathcal{C}=\frac{1}{4\pi}\int\limits_{M}\vert \na u \vert^2 \,d\sigma=\frac{1}{4\pi}\int\limits_{\partial M}\vert \na u \vert\, d\sigma\,.$$
Then, we have the following conclusions:
\begin{itemize}
\item[$(a)$] (Boundary gradient integral estimate) There holds
$$\int_{\partial M} \vert \nabla u \vert^{2}\, d\sigma\leq \pi.$$
If the equality holds, then $(M,g)$ is isometric to the Schwarzschild manifold of mass $\mathcal{C}$, 
\begin{equation}\label{feq63bisbis}
\Big(\R^{3}\setminus B_{\frac{ \mathcal{C}}{2}}(0),\,\Big(1+\frac{\mathcal{C}}{\,2\vert  x\vert}\Big)^{4}g_{{\mathrm{eucl}}}\Big)\,.
\end{equation}
\item [$(b)$] For all $t\in[\mathcal{C}/2, + \infty)$,
$$\frac{t^2}{\mathcal{C}^2} \,\left(1+\frac{\mathcal{C}}{2t}\right)^{\!4}\!\!\!\! \int\limits_{\Big\{u=\frac{1-\frac{\mathcal{C}}{2t}}{1+\frac{\mathcal{C}}{2t}}\Big\}}\!\!\!\vert \nabla u \vert^{2}\, d\sigma\leq 4\pi\,.$$
If the equality holds for some $t_0\in[\mathcal{C}/2, + \infty)$, then $(M,g)$ is isometric to the Schwarzschild manifold of mass $\mathcal{C}$, see \eqref{feq63bisbis}.
\item [$(c)$] (Area comparison) For every $t\in [\mathcal{C}/2, + \infty)$,
\begin{equation}\label{areacompwithboundary}
\mathrm{Area}\left(\Biggl\{u=\frac{1-\frac{\mathcal{C}}{2t}}{1+\frac{\mathcal{C}}{2t}}\Biggr\}^{\!\!*}\right)\geq 4\pi\,t^{2}\,\bigg(\!1+\frac{\mathcal{C}}{2t}\bigg)^{\!4}\,,
\end{equation}
where we adopted the convention that, if $\Sigma$ is a level set of $u$, then $\Sigma^*$ denotes $\Sigma\setminus\{\vert\nabla u\vert=0\}$.
If the equality holds for some $t_0\in [\mathcal{C}/2, + \infty)$, then $(M,g)$ is isometric to the Schwarzschild manifold of mass $\mathcal{C}$, see \eqref{feq63bisbis}.
\item [$(d)$] (Area--capacity inequality) We have
\begin{equation}
\sqrt{\frac{\mathrm{Area}(\partial M)}{16\pi}}\geq\mathcal{C}\,.
\end{equation}
If the equality is true, then $(M,g)$ is isometric to a Schwarzschild manifold of $\mathcal{C}$, see \eqref{feq63bisbis}.
\item[$(e)$] (Volume comparison) For any $t\in [\mathcal{C}/2, + \infty)$,
\begin{equation}
\mathrm{Vol}\left(\Biggl\{u\leq\frac{1-\frac{\mathcal{C}}{2t}}{1+\frac{\mathcal{C}}{2t}}\Biggr\}\right)\geq
 \int\limits_{\mathcal{C}/2}^t 4\pi\,s^{2}\,\bigg(\!1+\frac{\mathcal{C}}{2s}\bigg)^{\!6}\,ds\,.
\end{equation}
If the equality holds for some $t_0\in (\mathcal{C}/2, + \infty)$, then $(M,g)$ is isometric to the Schwarzschild manifold of mass $\mathcal{C}$, see \eqref{feq63bisbis}.
\end{itemize}
\end{theorem}

For completeness, we will present (as a consequence of the work \cite{Col_Min_5}) a sharp area and volume comparison, in the case that $(M,g)$ has no boundary, replacing the function $u$ with one related to the minimal positive Green's function. This volume comparison leads to a more geometric proof of the positive mass inequality than the one given in \cite{Ago_Maz_Oro_2}. The positive mass inequality is one part of the positive mass theorem, originally proven by Schoen--Yau \cite{SchYau79,SchYau81} and Witten \cite{Witten} (see also \cite{ParTau}).

\section{Non--compact $3$--manifolds with non--negative scalar curvature}

Let $(M,g)$ be a $3$--dimensional, complete, one--ended Riemannian manifold, whose boundary is either empty or minimal, compact and connected. Assume that $H_2(M,\partial M;\Z)=0$ and the scalar curvature of $(M,g)$ is non--negative.  
In this section, some gradient estimates, area and volume comparisons will be introduced.
To do this, we distinguish the cases when $M$ is without a boundary and when $M$ has a minimal, compact and connected boundary. 

\subsection{Boundaryless case}\label{Without boundary case}
In this case, we suppose that there exists the minimal positive Green's function $\mathcal{G}_{o}$ (for $\Delta$ operator) with pole $o\in M$, and that $\mathcal{G}_{o}, |\nabla \mathcal{G}_{o}|$ vanish at infinity. 
We consider the function $u=1-4\pi\, \mathcal{G}_{o}$ and we define the function
\begin{equation}\label{defhatF}
\widehat{F}(t)=\,-\, 4\pi \,t^{-1}\,+\,t\!\!\!\int\limits_{\{u=1-t^{-1}\}} \!\!\!\!\vert  \nabla u\vert^{2}\,d\sigma
\end{equation}
for every $t\in (0,+\infty)$. 
Here, $\sigma$ denotes the $2$--dimensional Hausdorff measure of $(M,g)$.
 Munteanu--Wang proved in \cite{MuntWang} that the function $\widehat{F}$ is non--decreasing and recently Colding--Minicozzi showed in \cite{Col_Min_5} that it is non--positive.
This last property implies the following area comparison
\begin{equation}\label{areacompwithoutboundary}
\mathrm{Area}\Big(\{u=1-t^{-1}\}^{*}\Big)=\mathrm{Area}\Big(\{u=1-t^{-1}\}\setminus \{\vert\nabla u\vert=0\}\Big)\geq 4\pi\,t^2\,,
\end{equation}
which in turn gives 
\begin{equation}\label{volumecompwithoutboundary}
\mathrm{Vol}\Big(\{u\leq 1-t^{-1}\}\Big)\geq\frac{4}{3}\pi\,t^3\,.
\end{equation}
Indeed, since the equality 
\begin{equation}\label{idcap}
\int\limits_{\{u=1-t^{-1}\}} \!\!\!\!\vert  \nabla u\vert\,d\sigma=4\pi
\end{equation}
holds, we obtain
\begin{align*}
4\pi=\int\limits_{\{u=1-t^{-1}\}} \!\!\!\!\vert  \nabla u\vert\,d\sigma&\leq \Bigg(\,\,\int\limits_{\{u=1-t^{-1}\}} \!\!\!\!\vert  \nabla u\vert^{2}\,d\sigma\Bigg)^{\frac{1}{2}}\,\mathrm{Area}\Big(\{u=1-t^{-1}\}^*\Big)^{\frac{1}{2}}\\
&\leq \sqrt{4\pi}\,t^{-1}\,\mathrm{Area}\Big(\{u=1-t^{-1}\}^*\Big)^{\frac{1}{2}}\,,
\end{align*}
from which inequality \eqref{areacompwithoutboundary} follows. Moreover, if the equality in \eqref{areacompwithoutboundary} holds for some value $t_0\in (0,+\infty)$, then $(M,g)$ is isometric to the Euclidean space of dimension~ $3$. This is a consequence of the fact that if there exists $t_0\in (0,+\infty)$ such that $\widehat{F}(t_0)=0$, then $(M,g)$ is isometric to the $3$--dimensional Euclidean space, see \cite{Col_Min_5}, for instance.
Concerning the volume comparison \eqref{volumecompwithoutboundary}, we start observing that 
\begin{align}
\mathrm{Vol}\Big(\{u\leq 1-t^{-1}\}\Big)
= \int\limits_{0}^t\Bigg(\,\int\limits_{\{u\leq 1-s^{-1}\}} \!\!\!\!\frac{1}{\vert  \nabla u\vert}\,d\sigma\Bigg) s^{-2}\, ds \label{ffeq100}
\end{align}
by the coarea formula \cite[Proposition 2.1]{BMF}. Here, we are keeping into account that the set $\left\lbrace\abs{\nabla u}=0\right\rbrace$ is $\sigma$--negligible, see~\cite[Theorem~1.1]{Hardt2}. Now, since
\begin{align*}
4\pi s^2&\leq \mathrm{Area}\Big(\{u\leq 1-s^{-1}\}^*\Big)\\
&\leq \Bigg(\,\int\limits_{\{u\leq 1-s^{-1}\}} \!\!\!\!\frac{1}{\vert  \nabla u\vert}\,d\sigma\Bigg)^{\frac{2}{3}} \, \Bigg(\,\int\limits_{\{u\leq 1-s^{-1}\}} \!\!\!\!\vert  \nabla u\vert^2\,d\sigma\Bigg)^{\frac{1}{3}}\\
&\leq (4\pi s^{-2})^{\frac{1}{3}}\Bigg(\,\int\limits_{\{u\leq 1-s^{-1}\}} \!\!\!\!\frac{1}{\vert  \nabla u\vert}\,d\sigma\Bigg)^{\frac{2}{3}}\,,
\end{align*}
we have
$$\int\limits_{\{u\leq 1-s^{-1}\}} \!\!\!\!\frac{1}{\vert  \nabla u\vert}\,d\sigma\geq 4\pi \,s^{4}\,.$$
 Consequently, we obtain
\begin{align}
\mathrm{Vol}\Big(\{u\leq 1-t^{-1}\}\Big)\geq \int\limits_{0}^t 4\pi\, s^2\,ds=\frac{4}{3}\pi\, t^3\,.
\end{align}
Furthermore, if the equality holds for some value $t_0\in (0,+\infty)$, then $(M,g)$ is isometric to the $3$--dimensional Euclidean space.
\medskip

A relevant consequence of this volume comparison estimate is an alternative proof of the (Riemannian) positive mass inequality in dimension $3$. 
We also remark that some link between a volume comparison estimate
and the positive mass inequality is already known, see for instance \cite{FanShiTam, Jar}.

In order to state precisely this theorem, let us first recall the definition of one--ended asymptotically flat manifold and of ADM mass.

\begin{definition}\label{defAFman}
A $3$--dimensional Riemannian manifold $(M,g)$ (with or without boundary) is said to be {\em one--ended asymptotically flat} if there exists a closed and bounded subset $K$ and a diffeomorphism $\Phi: M\setminus K\to\R^{3}\setminus\overline{B}_{r}(0)$ such that in the (coordinate) chart $(M\setminus K, \Phi=(x^{i}))$, called {\em asymptotically flat (coordinate) chart}, setting $g=g_{ij}\,dx^{i}\otimes dx^{j}$, there holds
\begin{equation}\label{eq4}
g_{ij}=\delta_{ij} + O_{2}(\vert x \vert ^{-\tau})\,,
\end{equation}
for some constant $\tau>1/2$ (the {\em order of decay} of $g$ in the asymptotically flat chart $(M\setminus K, (x^{i}))$, briefly {\em the order}), where $\delta$ is the Kronecker delta function.
\end{definition}

\begin{definition}\label{defADMmass}
Let $(M,g)$ be a one--ended asymptotically flat manifold having integrable or non--negative scalar curvature and let $\big(E,(x^{1},x^{2},x^{3})\big)$ be an asymptotically flat chart.
The following limit
\begin{equation}\label{formADMmass}
m_{\mathrm{ADM}}=\lim_{r\to +\infty}\frac{1}{16\pi}\!\!\!\int\limits_{\{\vert  x\vert\,=\,r\}}\!\!\!(\partial_{j}g_{ij}-\partial_{i}g_{jj})\frac{x^{i}}{\vert  x \vert}\,d\sigma_{\mathrm{eucl}}\,,
\end{equation}
exists (possibly equal to $+\infty$) and it is independent of the asymptotically flat chart (proved first by Bartnik~\cite{Bartnik} and then independently by Chru\'sciel~\cite{Chrusciel}).
This geometric invariant is called the {\em ADM mass} of $(M,g)$.
\end{definition}

\begin{theorem}[Riemannian positive mass theorem]\label{PMT}
\label{PMT}
Let $(N,h)$ be a $3$--dimensional, complete, one--ended asymptotically flat manifold with non--negative scalar curvature. Then, the ADM mass of $(N,h)$ is non--negative, that is
\begin{equation*}
m_{ADM}(N,h) \geq 0.
\end{equation*}
Moreover, $m_{ADM}(N,h)=0$ if and only if $(N,h)$ is isometric to $(\R^{3},g_{\R^3})$.
\end{theorem}
\begin{proof}
The first part is the same as that given in \cite{Ago_Maz_Oro_2}. 
Let us focus on the inequality $m_{ADM}(N,h) \geq 0$.
We use~\cite[Section 2]{bray3} to reduce the analysis to the case where the underlying manifold $N$ is diffeomorphic to $\R^3$ and there exists a distinguished asymptotically flat coordinate chart $x=(x^1,x^2,x^3)$ in which the metric $g$ can be expressed as 
\begin{equation}
\label{eq:schwas}
g \, = \, \Big(1+\frac{m_{ADM}}{2\vert x\vert}\Big)^{\!4}\,\delta_{ij} \,\, dx^{i}\!\otimes dx^{j}.
\end{equation}
 In a such chart, the function $u$ has the following asymptotic expansion:
\begin{align}
u& \, = \, 1-\frac1{\vert x\vert}+\frac1{2\vert x\vert^{2}}\,
\big(m_{ADM}+\phi (x/\vert x\vert)\big)+O_2\big(\vert x\vert^{-3+	\alpha} \big) \, ,\label{expansionofu}
\end{align}
where $\phi$ satisfies $\Delta_{\SSS^{2}}\phi=-2\phi$ and $0<\alpha<1/2$ is a fixed real number that can be chosen as small as desired. Moreover, we note that a manifold of this kind clearly satisfies all the assumptions made at the beginning of this section. Therefore, we rewrite the equality \eqref{ffeq100} in the following way,
\begin{equation*}
\mathrm{Vol}\Big(\big\{u\leq 1-\,t^{-1}\big\}\Big)=\int\limits_{0}^t\Bigg(\,\int\limits_{\{u=1-\,s^{-1}\}} \!\!\!\!\frac{(1-u)^4}{\vert  \nabla u\vert^2}\,\vert  \nabla u\vert \,d\sigma\Bigg) s^{2}\, ds.
\end{equation*}
Now, we observe that
\begin{equation}
\frac{(1-u)^4}{\vert  \nabla u\vert^2}=\,1+ 2m_{ADM}\,\vert x\vert^{-1}+O\big(\vert x\vert^{-2+\alpha}\big)\,, 
\end{equation}
as
\begin{equation}
\vert \nabla u\vert = \, \vert x\vert^{-2}\big[\,1-\big(2m_{ADM}+\phi (x/\vert x\vert)\big)\vert x\vert^{-1}+O\big(\vert x\vert^{-2+\alpha}\big)\,\big]
\end{equation}
due to the asymptotic expansion \eqref{expansionofu} and the expression of metric \eqref{eq:schwas}. Then, using identity \eqref{idcap}, we obtain
\begin{equation}
\mathrm{Vol}\Big(\big\{u\leq 1-\,t^{-1}\big\}\Big)=\frac{4}{3}\pi\, t^3+4\pi m_{ADM}\,t^2+o(t^2)\,,
\end{equation}
which implies
\begin{equation}
0\leq t^{-2}\,\Big[\mathrm{Vol}\Big(\big\{u\leq 1-\,t^{-1}\big\}\Big)-\frac{4}{3}\pi\, t^3\Big]=4\pi m_{ADM}+o_t(1)\,,
\end{equation}
by the volume comparison \eqref{volumecompwithoutboundary}. The above inequality clearly leads to the wanted inequality, $m_{ADM}\geq 0$.

The rigidity statement, i.e. the fact that $(N,h)$ and $(\R^{3},g_{\R^3})$ are isometric if $m_{ADM}=0$, can be deduced from the validity of the inequality $m_{ADM}\geq 0$, through the nowadays standard argument proposed in the original Schoen--Yau's paper \cite{SchYau79} (see also~\cite[pp. 95-97 and p. 102]{DanLee}).
\end{proof}

Such proof uses the same tools of the one given in \cite{Ago_Maz_Oro_2}. Furthermore, if on one hand it is more geometric, on the other hand, it is not clear if it works without simplifying the metric at infinity (this is actually true for the proof in \cite{Ago_Maz_Oro_2}). Most likely, such an assumption cannot be removed: one reason is that Jauregui \cite{Jar} showed that for a generic asymptotically flat chart $(x^1,x^2,x^3)$, there holds
\begin{align}
V (r)&= \text{$g$--volume of the compact region enclosed by $\{|x|=r\}$}\\
&=\frac{4}{3}\pi r^3 +2\pi m_{\mathrm{ADM}} r^2+2\pi r\beta(r)+o(r^2)
\end{align}
where $\beta(r)=2m_{\mathrm{ADM}} r+O(1)$ in a harmonically flat manifold, while in general the $r\beta(r)$ term could be problematic.
Moreover, in our case the sub--level sets of the function $u$ should replace the compact regions enclosed by the coordinate spheres.

\subsection{Minimal, compact and connected boundary case}
In this case, we suppose that there exists the smooth solution $u$ of the problem
\begin{equation}\label{cappot}
\Delta u=0\ \,\, \mathrm{in} \ {M}\quad\quad\quad u=0\ \,\,\mathrm{on} \  \partial M \quad\quad\quad u \to 1\,\, \mathrm{at} \  \infty\,.
\end{equation}
Then the boundary capacity $\mathcal{C}$ of $\partial M$ in $(M,g)$, defined by
$$\mathcal{C}\,=\,\inf\Bigg\{\frac{1}{4\pi}\int\limits_{M} \vert \nabla \psi \vert^2\, d\sigma\,:\, \psi\in C^{\infty}_c(M)\quad\text{and}\quad \psi=1\quad \text{on}\ \partial M\Bigg\}\,,$$
can be expressed in terms of the function $u$ as 
\begin{equation}\label{boundcapintermofu}
\mathcal{C}=\frac{1}{4\pi}\int\limits_{M}\vert \na u \vert^2 \,d\mu=\frac{1}{4\pi}\int\limits_{\partial M}\vert \na u \vert\, d\sigma\,.
\end{equation}
Let $G$ be the function defined on the interval $[\mathcal{C}/2,+\infty)$ given by formula
\begin{equation}\label{defG}
G(t)= -\,\frac{\pi\mathcal{C}^{2}}{t}+\frac{t}{4} \,\left(1+\frac{\mathcal{C}}{2t}\right)^{\!4} \int\limits_{\Sigma_{t}}\vert \nabla u \vert^{2}\, d\sigma\,,
\end{equation}
where $\Sigma_{t}$ denotes the following level set of $u$, 
\begin{equation}\label{Sigmat}
\Sigma_{t}=\Big\{u=\Big(1-\frac{\mathcal{C}}{2t}\Big)/\Big(1+\frac{\mathcal{C}}{2t}\Big)\Big\}\,.
\end{equation}
The function $G$ is locally Lipschitz, since the function
$$t\in [\mathcal{C}/2,+\infty)\mapsto \int\limits_{\Sigma_{t}}\vert \nabla u \vert^{2}\, d\sigma\in [0,+\infty)$$
 is locally Lipschitz by \cite[Lemma 11 and Lemma 12]{ChodLi}.
 By \cite[Proposition 4.1]{Or_1}, we have
 \begin{equation}\label{G'}
G'(t)=\frac{\pi\mathcal{C}^{2}}{t^{2}} +\frac{1}{4}\,\left(1+\frac{\mathcal{C}}{2t}\right)^{\!3}\left(1-\frac{3\mathcal{C}}{2t}\right)\int\limits_{\Sigma_{t}}\vert \nabla u \vert^{2}\, d\sigma-\frac{\mathcal{C}}{4t}\left(1+\frac{\mathcal{C}}{2t}\right)^{\!2}\!\int\limits_{\Sigma_{t}}\vert \nabla u \vert\,\mathrm{H}\, d\sigma\,
\end{equation}
for every $t\in\mathcal{T}$, where $\mathcal{T}$ is the open set of $ [\mathcal{C}/2,+\infty)$ given by
\begin{equation}\label{defTcap}
\mathcal{T}=\Big\{t\in[\mathcal{C}/2,+\infty)\,:\, \Big(1-\frac{\mathcal{C}}{2t}\Big)/\Big(1+\frac{\mathcal{C}}{2t}\Big) \text{ is a regular value of $u$}\Big\}\,
\end{equation}
(we recall that $\partial M= \{u = 0\}$, by the maximum principle, and that zero is a regular value of $u$, by the Hopf lemma).
Moreover, following an analogous argument to that one of \cite[Theorem 1.1]{Ago_Maz_Oro_2} and using in it the proof of \cite[Proposition 3.1]{Or_1}, the function $F:[\mathcal{C}/2,+\infty)\to\R$, defined as 
\begin{equation}\label{feq40}
F(t)=4\pi t\,+ \,\frac{t^{3}}{\mathcal{C}^{2}} \,\left(1+\frac{\mathcal{C}}{2t}\right)^{\!3} \!\left(1-\frac{3\mathcal{C}}{2t}\right)\int\limits_{\Sigma_{t}} \vert \nabla u \vert^{2}\, d\sigma\,-\,\frac{t^{2}}{\mathcal{C}} \,\left(1+\frac{\mathcal{C}}{2t}\right)^{\!2}\!\int\limits_{\Sigma_{t}}\vert \nabla u \vert\,\mathrm{H}\, d\sigma
\end{equation}
and satisfying the equality
\begin{equation}\label{eqf50}
F(t)=\,\frac{4t^{3}}{\mathcal{C}^{2}}\, G'(t)
\end{equation}
for every $t\in \mathcal{T}$, belongs to $W^{1,1}_{loc}(\mathcal{C}/2,+\infty)$. Above, $\mathrm{H}$ is the mean curvature of $\Sigma_{t}^*$ with respect to the $\infty$--pointing unit normal vector field $\nu={\nabla u}/{\vert \nabla u \,\vert}$, expressed in terms of $u$ as $-g(\nabla |\nabla u|, \nabla u)/|\nabla u|^2$.
It then follows that $G\in W^{2,1}_{loc}(\mathcal{C}/2,+\infty)$. Thus, it is of class $C^1([\mathcal{C}/2, + \infty))$ and its first derivative is locally absolutely continuous on the interval $[\mathcal{C}/2, + \infty)$.
\smallskip

Following an analogous argument to the one of \cite[Proposition 0.4]{Col_Min_5}, we will prove that the function $G$ is both non--decreasing and non--positive, under the additional assumption that $|\nabla u|$ vanishes at infinity.
To obtain this, let us set
\begin{equation}
A=F(\mathcal{C}/2)=2\mathcal{C}\Big(\pi-\int_{\partial M} \vert \nabla u \vert^{2}\, d\sigma\Big)\,,
\end{equation}
and we consider the functions 
\begin{align}
A_1(t)&=\frac{t^2}{\mathcal{C}^2} \,\left(1+\frac{\mathcal{C}}{2t}\right)^{\!4} \int\limits_{\Sigma_{t}}\vert \nabla u \vert^{2}\, d\sigma\,\label{defA_1}\\
\widetilde{A}_1(t)&=A_1(t)+\frac{1}{2t}\,A\label{tildadefA_1}
\end{align}
for every $t\in[\mathcal{C}/2, + \infty)$.
Being
$$A_1(t)=4\pi+\frac{4t}{\mathcal{C}^2}\, G(t)$$
for each $t\in [\mathcal{C}/2, + \infty)$, it turns out that $A_1$ is of class $C^1([\mathcal{C}/2, + \infty))$ and its first derivative is locally absolutely continuous in the interval $[\mathcal{C}/2, + \infty)$.
Obviously, the same results hold for the function $\widetilde{A}_1$.
We will show that $A_1(t)\leq 4\pi$, if $A\geq 0$, and $A_1(t)\leq 4\pi-A/t$, otherwise.
These inequalities will be a consequence the following quadratic lower bound for the growth of $A_1$, which holds without the additional assumption: $|\nabla u|\to 0$ at infinity.

\begin{proposition}\label{PropCol_Min}
Let $(M,g)$ be a $3$--dimensional, complete, one--ended Riemannian manifold, whose boundary is minimal, compact and connected. 
Assume that the scalar curvature of $(M,g)$ is non--negative. We suppose that there exists the smooth solution $u$ to problem \eqref{cappot}, and that all regular level sets of $u$ are connected.\\
If $A\geq 0$ and $A_1(t_0)>4\pi$ for some $t_0\in [\mathcal{C}/2, + \infty)$, then there exists a positive constant $C$ such that, for every $t\geq t_0$,
$$A_1(t)\geq C\,t^2\,.$$
If $A< 0$ and $A_1(t_0)>4\pi-A/t_0$ for some $t_0\in [\mathcal{C}/2, + \infty)$, then there exists a positive constant $C$ such that, for all $t\geq t_0$,
$$A_1(t)\geq C\,t^2\,.$$
\end{proposition}

\begin{proof}
We introduce the $(0,2)$--symmetric tensors of $M$,
\begin{equation}
T=\nabla du+\frac{6u}{1-u^2} \,du\otimes du\quad \quad \text{and}\quad\quad B=T-\frac{\mathrm{tr}(T)}{3}\,g\,,
\end{equation}
and we define
\begin{equation}
S_1(t)=\int\limits_{\Sigma_t} {\rm R}\,d\sigma\quad\quad \text{and}\quad\quad B_1(t)=\int\limits_{\Sigma_t} \frac{|B|^2}{|\nabla u|^2}\,d\sigma\,.
\end{equation}
on $[\mathcal{C}/2, + \infty)$ and $\mathcal{T}$, respectively. Let us divide the proof into steps.

\textbf{Step 1:} {\em We have as follows.
\begin{itemize}
\item [$(1)$] If $A\geq 0$,  
$$tA_1'(t)\geq A_1(t)-4\pi+\frac{1}{2t}\int\limits_{\mathcal{C}/2}^t \big(R_1(s)+B_1(s)\big)\,ds$$
holds everywhere.
\item [$(2)$] For all $t\in[\mathcal{C}/2, + \infty)$,
$$t\widetilde{A}_1'(t)\geq \widetilde{A}_1(t)-4\pi+\frac{1}{2t}\int\limits_{\mathcal{C}/2}^t \big(R_1(s)+B_1(s)\big)\,ds\,.$$
\end{itemize}
}
The starting point is observing that, for every $t\in\mathcal{T}$,
\begin{equation}\label{fffeq1}
A_1'(t)=\frac{2t}{\mathcal{C}^2}\,\left(1+\frac{\mathcal{C}}{2t}\right)^{\!3}\left(1-\frac{\mathcal{C}}{2t}\right)\int\limits_{\Sigma_{t}}\vert \nabla u \vert^{2}\, d\sigma\,-\,\frac{1}{\mathcal{C}} \,\left(1+\frac{\mathcal{C}}{2t}\right)^{\!2}\!\int\limits_{\Sigma_{t}}\vert \nabla u \vert\,\mathrm{H}\, d\sigma\,,
\end{equation}
being
\begin{equation}\label{cap4geq18}
\frac{d}{dt}\int\limits_{\Sigma_{t}} \vert \nabla u \vert^{2}\, d\sigma=\,-\,\frac{\mathcal{C}}{t^{2}}\left(1+\frac{\mathcal{C}}{2t}\right)^{\!-2}\int\limits_{\Sigma_{t}} \vert \nabla u \vert\,\mathrm{H}\, d\sigma
\end{equation}
 by the normal first variation of the volume measure.
Now, equality \eqref{fffeq1} implies 
\begin{equation}\label{fffeq3}
tA_1'(t)-A_1(t)=\frac{1}{t}\,\Big(F(t)-4\pi\,t\Big)=\frac{1}{t}\Big(A+\int\limits_{\mathcal{C}/2}^t F'(s)\,ds\Big)\,-\,4\pi\,,
\end{equation}
for all $t\in \mathcal{T}$. Here, in the second equality we used the fact that $F$ admits a locally absolutely continuous representative in $[\mathcal{C}/2, + \infty)$ coinciding with it in the set $\mathcal{T}$ 
and that zero is a regular level set of $u$.
Then, we recall that 
\begin{equation}\label{feq54}
F'(t)=4\pi- \int\limits_{\Sigma_{t}}\frac{\,\rm{R}^{\Sigma_{t}}}{2} \, d\sigma +\int\limits_{\Sigma_{t}}\bigg[\frac{\vert \nabla^{\Sigma_{t}}\vert \nabla u\vert\vert^{2}}{\vert \nabla u \vert^{2}}+\frac{\RRR}{2}+\frac{\,\vert \ringg{\mathrm{h}}\vert^{2}}{2}+ \frac{3}{4}\left(\frac{4u}{1-u^{2}}\,\,\vert \nabla u \vert-\mathrm{H}\right)^{\!2}\bigg]d\sigma
\end{equation}
for every $t\in \mathcal{T}$, since
\begin{align}
\frac{d}{dt}\int\limits_{\Sigma_{t}} \vert \nabla u \vert \,\mathrm{H}\,d\sigma&=-\frac{\mathcal{C}}{t^{2}}\left(1+\frac{\mathcal{C}}{2t}\right)^{\!-2} \int\limits_{\Sigma_{t}} \vert \nabla u \vert \left[ \, \Delta^{\Sigma_{t}} \!\left(\frac{1}{\vert \nabla u \vert}\right) + \, \frac{\vert \mathrm{h}\vert^{2}+\Ric (\nu,\nu)}{\vert \nabla u\vert} \, \right]\,d\sigma\,\\
&=-\frac{\mathcal{C}}{t^{2}}\left(1+\frac{\mathcal{C}}{2t}\right)^{\!-2}\int\limits_{\Sigma_{t,\mathcal{C}}} \bigg[ \frac{\vert \nabla^{\Sigma_t}\vert \nabla u\vert\vert^{2}}{\vert \nabla u \vert^{2}}+\frac{\RRR}{2}-\frac{\RRR^{\Sigma_{t}}}{2}+\frac{\,\vert \ringg{\mathrm{h}}\vert^{2}}{2}+\frac{3\,\mathrm{H}^{2}}{4}\bigg]\,d\sigma\,,
\end{align}
where $\na^{\Sigma_{t}}$, $\Delta^{\Sigma_{t}}$ are the Levi--Civita connection and the Laplace--Beltrami operator of the induced metric $g^{\Sigma_{t}}$, respectively, $\rm{R}^{\Sigma_{t}}$ is the scalar curvature of ${\Sigma_{t}}$ and finally, $\mathrm{h}$, $\ringg{\mathrm{h}}$ denote the second fundamental form of ${\Sigma_{t}}$ and its traceless part with respect to $\nu=\na u /\vert\na u\vert$. In the above, the first equality follows from the first normal variation of the volume measure and of the mean curvature, whereas the last one is obtained with the help of the traced Gauss equation and the divergence theorem.\\
By Sard's theorem, the closed set $[\mathcal{C}/2, + \infty)\setminus\mathcal{T}$ has zero Lebesgue measure. Thus, putting together equations \eqref{fffeq3} and \eqref{feq54}, by Gauss--Bonnet theorem and continuity, we get 
\begin{align}
tA_1'(t)-A_1(t)+4\pi&=\frac{1}{t}\Big(A+\int\limits_{\mathcal{C}/2}^t F'(s)\,ds\Big)\\
&\geq \frac{A}{t}+\frac{1}{t}\int\limits_{\mathcal{C}/2}^t \int\limits_{\Sigma_{s}}\bigg[\frac{\vert \nabla^{\Sigma_{s}}\vert \nabla u\vert\vert^{2}}{\vert \nabla u \vert^{2}}+\frac{\RRR}{2}+\frac{\,\vert \ringg{\mathrm{h}}\vert^{2}}{2}+ \frac{3}{4}\left(\frac{4u}{1-u^{2}}\,\,\vert \nabla u \vert-\mathrm{H}\right)^{\!2}\bigg]d\sigma\,ds\,, \label{fffeq4}
\end{align}
for every $t\in [\mathcal{C}/2, + \infty)$.
At the same time, on each regular level set $\Sigma$ of $u$, the following identities
\begin{align}
\vert  \nabla du \vert^{2}&=\vert \nabla u \vert^{2} \vert  \mathrm{h} \vert^{2}+\vert\nabla\vert \nabla u \vert\vert^{2}+\vert\nabla^{\Sigma}\vert \nabla u \vert\vert^{2}\\
\vert \mathrm{H}\vert^{2}&=\frac{\vert\nabla^{\perp}\vert \nabla u \vert\vert^{2}}{\vert \nabla u \vert^{2}}
\end{align}
imply
\begin{align}
\frac{\vert \nabla^{\Sigma}\vert \nabla u\vert\vert^{2}}{\vert \nabla u \vert^{2}}&+\frac{\,\vert \ringg{\mathrm{h}}\vert^{2}}{2}+ \frac{3}{4}\left(\frac{4u}{1-u^{2}}\,\,\vert \nabla u \vert-\mathrm{H}\right)^{\!2}\\
&=\frac{\vert \nabla^{\Sigma}\vert \nabla u\vert\vert^{2}}{\vert \nabla u \vert^{2}}+\frac{\,\vert \mathrm{h}\vert^{2}}{2}+ \frac{ \mathrm{H}^2}{2}+\frac{12 u^{2}}{(1-u^{2})^{2}}\,\vert \nabla u \vert^{2}
+\frac{6 u }{ 1-u^{2}} \, \frac{g(\nabla \vert \nabla u\vert,\nabla u )}{\vert \nabla u\vert}\\
&=\frac{\vert \nabla du\vert^2}{2\vert \nabla u \vert^{2}}+\frac{12 u^{2}}{(1-u^{2})^{2}}\,\vert \nabla u \vert^{2}
+\frac{6 u }{ 1-u^{2}} \, \frac{g(\nabla \vert \nabla u\vert,\nabla u )}{\vert \nabla u\vert}=\frac{\vert B\vert^2}{2\vert \nabla u \vert^{2}}\,. \label{fffeq5}
\end{align}
Thus, by inequality \eqref{fffeq4} and equality \eqref{fffeq5}, we get
\begin{equation}
tA_1'(t)-A_1(t)+4\pi\geq \frac{A}{t}+\frac{1}{2t}\int\limits_{\mathcal{C}/2}^t \big(R_1(s)+B_1(s)\big)\,ds\,,
\end{equation}
leading to
\begin{equation}
t\widetilde{A}_1'(t)-\widetilde{A}_1(t)+4\pi\geq \frac{1}{2t}\int\limits_{\mathcal{C}/2}^t \big(R_1(s)+B_1(s)\big)\,ds\,.
\end{equation}
The last two inequalities then imply the desired ones.\\

\textbf{Step 2}: {\em For every $t\in \mathcal{T}$, we have
\begin{equation}
tA_1'(t)=\frac{t}{\mathcal{C}}\,\left(1+\frac{\mathcal{C}}{2t}\right)^{\!2} \int\limits_{\Sigma_{t}} B(\nu,\nu)\,d\sigma\,,
\end{equation}
where $\nu=\na u /\vert\na u\vert$.\\
} 
This equality follows trivially from formula \eqref{fffeq1}, as $\mathrm{H}=-\nabla du (\nabla u, \nabla u)/\vert \nabla u \vert^3$.\\

\textbf{Step 3}: {\em There exists a positive constant $C$ such that $A_1(t)\geq C\,t^2\,$ for all $t\geq t_0$.}\\
We need to distinguish the cases: $A\geq 0$ and $A_1(t_0)>4\pi$ for some $t_0\in [\mathcal{C}/2, + \infty)$; $A<0$ and $A_1(t_0)>4\pi-A/t_0$ for some $t_0\in [\mathcal{C}/2, + \infty)$, which imply $\widetilde{A_1}(t_0)>4\pi$.
However, in both cases, we will show that there exists $\delta>0$ such that, for every $t\geq t_0$,
\begin{equation}\label{fffeq6}
A'_1(t)\geq 0 \quad\quad \text{and}\quad\quad A_1(t)\geq 4\pi+\delta\,\frac{t}{t_0}\,.
\end{equation}
Let $\delta>0$ be such that $A_1(t_0)>4\pi+\delta$ if $A\geq 0$, otherwise $\widetilde{A_1}(t_0)>4\pi+\delta$. Then, taking advantage of the hypothesis $\mathrm{R}\geq 0$, Step 1 gives the following.
\begin{itemize}
\item If $A\geq 0$, then $A'_1(t)\geq 0 $ and $A_1(t)\geq 4\pi+\delta\,(t/t_0)$ for every $t\geq t_0$;
\item If $A<0$, then $A_1(t)>4\pi-A/t_0\geq 4\pi-A/t$ for every $t\geq t_0$, which gives $A'_1(t)\geq 0$ in the interval $[t_0,+\infty)$. Moreover, since it turns out that $\widetilde{A}_1(t)\geq 4\pi+\delta\,(t/t_0)$ for all $t\geq t_0$, and recalling that $\widetilde{A}_1(t)=A_1(t)+ A/(2t)$ with $A<0$, it then follows that $A_1(t)\geq 4\pi+\delta\,(t/t_0)$ for every $t\geq t_0$.
\end{itemize}
As in \cite{Col_Min_5}, let us consider on the interval $[\mathcal{C}/2,+\infty)$ the function
\begin{equation}\label{fffeq12}
a(t)=t\,\frac{A'_1(t)}{A_1(t)}
\end{equation}
that measures the rate of the polynomial growth of $A_1$.
Since the function $A_1$ is of class $C^1([\mathcal{C}/2, + \infty))$ and its first derivative is locally absolutely continuous on the interval $[\mathcal{C}/2, + \infty)$, from inequality \eqref{fffeq6} it follows that the function $a(t)$ is locally absolutely continuous on $[\mathcal{C}/2, + \infty)$. Notice that $a(t)\geq 0$ for all $t\geq t_0$, by \eqref{fffeq6}.
 Now, for every $t\in\mathcal{T}$, observing that the first equality of \eqref{fffeq3}, that is $tA_1'(t)-A_1(t)=F(t)/t-4\pi$, leads to 
\begin{align}
A_1'(t)&=\frac{F(t)}{t^2}+\frac{1}{t}\,(A_1(t)-4\pi)\\
tA_1''(t)&=[tA_1'(t)-A_1(t)]'=-\frac{F(t)}{t^2}\, +\frac{F'(t)}{t}\,,
\end{align}
it follows
\begin{equation}\label{fffeq8}
a'(t)=\frac{1}{t}\,\Big(1-\frac{4\pi}{A_1(t)}+\frac{F'(t)}{A_1(t)}-a^2(t)\Big)\,.
\end{equation}
Now, let us observe, similarly to Colding--Minicozzi in \cite{Col_Min_5}, that
\begin{equation}\label{fffeq7}
\frac{F'(t)}{A_1(t)}\geq \frac{3}{4}\,a^{2}(t)\,.
\end{equation}
Indeed, Step 2 and the Cauchy--Schwarz inequality give
\begin{equation}
\big(t\,A'_1(t)\big)^2\leq A_1(t)\int\limits_{\Sigma_{t}} \Big(\frac{B(\nu,\nu)}{\vert \nabla u\vert}\Big)^{2}\,d\sigma\,,
\end{equation}
which along with \cite[Lemma 2.1]{Col_Min_5} implies
\begin{equation}
\big(t\,A'_1(t)\big)^2\leq \frac{2}{3}\,A_1(t)\, B_1(t)\,.
\end{equation}
Putting together the Gauss--Bonnet theorem, the assumption $\mathrm{R}\geq 0$ and the identity \eqref{fffeq5} in equality \eqref{feq54}, it follows $F'(t)\geq B_1(t)/2$, which leads to inequality \eqref{fffeq7}.
By using such inequality in equality \eqref{fffeq8}, there holds 
\begin{equation}\label{fffeq9}
a'(t)\geq \frac{1}{t}\,\Big(1-\frac{4\pi}{A_1(t)}-\frac{1}{4}\,a^2(t)\Big)\,,
\end{equation}
for every $t\in \mathcal{T}$.\\
We are now able to show the existence of a positive constant $C=C(t_0,\delta)$ such that
\begin{equation}
a(t)\geq 2-\frac{C}{\sqrt{t}}
\end{equation}
for all $t\geq t_0$. 
If $a(t)\geq 2$ for all $t\geq t_0$, the claim is trivially true. Therefore, we suppose that there exists $\widetilde{t}_0\geq t_0$ such that $a(\,\widetilde{t}_0)<2$. 
Without loss of generality, we can assume that $\widetilde{t}_0> t_0$ by the continuity of $a(t)$, and that $\widetilde{t}_0\in\mathcal{T}$ by Sard's theorem.
Then, let $t_2>t_0$ be a generic element of $\mathcal{T}$ such that $a(t_2)<2$.
Let us denote by $t_1$ the infimum of the set of $t\in[t_0, t_2]$ such that the function $a(t)$ is strictly less than $2$ on the interval $(t,t_2]$.
Either $t_1=t_0$ or $t_1>t_0$.
In the first case, $0\leq a(t)\leq 2$ for every $t\in [t_0,t_2]$, therefore from inequality \eqref{fffeq9}, it follows that
\begin{equation}
a'(t)\geq \frac{1}{t}\,\Big(1-\frac{4\pi}{A_1(t)}-\frac{1}{2}\,a(t)\Big)\,,
\end{equation}
which implies
\begin{equation}\label{fffeq10}
\big(\sqrt{t}\,a(t)\big)'\geq \frac{1}{\sqrt{t}}\,\Big(1-\frac{4\pi}{A_1(t)}\Big)\,,
\end{equation}
for every $t\in \mathcal{T}\cap [t_0,t_2]$.
The function $\sqrt{t}\,a(t)$ is locally absolute continuous in $[\mathcal{C}/2, + \infty)$, thus, by integrating this inequality from $t_0$ to $t_2$, we obtain
\begin{equation}
a(t_2)\geq 2+\sqrt{\frac{t_0}{t_2}} \,a(t_0)-2\,\sqrt{\frac{t_0}{t_2}}-\frac{4\pi}{\sqrt{t_2}}\int\limits_{t_0}^{t_2} \frac{ds}{\sqrt{s}\,A_1(s)}\,.
\end{equation}
Then, it follows from the second inequality \eqref{fffeq6} that
\begin{equation}
a(t_2)\geq 2+\sqrt{\frac{t_0}{t_2}} \,a(t_0)-2\,\sqrt{\frac{t_0}{t_2}}-\frac{8\pi}{\delta}\,\sqrt{\frac{t_0}{t_2}}\geq  2-2\Big(1+\frac{4\pi}{\delta}\Big)\,\sqrt{\frac{t_0}{t_2}}\,.
\end{equation}
In the other case, we have $a(t_1)=2$, in addition to having $0\leq a(t)\leq 2$ for every $t\in [t_1,t_2]$. Thus, proceeding as above, we have
\begin{align}
a(t_2)&\geq 2+\sqrt{\frac{t_1}{t_2}} \,a(t_1)-2\,\sqrt{\frac{t_1}{t_2}}-\frac{4\pi}{\sqrt{t_2}}\int\limits_{t_1}^{t_2} \frac{ds}{\sqrt{s}\,A_1(s)}=2-\frac{4\pi}{\sqrt{t_2}}\int\limits_{t_1}^{t_2} \frac{ds}{\sqrt{s}\,A_1(s)}\\
&\geq2-\frac{4\pi}{\sqrt{t_2}}\int\limits_{t_0}^{t_2} \frac{ds}{\sqrt{s}\,A_1(s)}\geq 2-\frac{8\pi}{\delta}\,\sqrt{\frac{t_0}{t_2}} \,.
\end{align}
Consequently, choosing $C=C(\delta,t_0)=2\big(1+\frac{4\pi}{\delta}\big)\sqrt{t_0}$, we get
\begin{equation}\label{fffeq11}
a(t)\geq 2-\frac{C}{\sqrt{t}}\,,
\end{equation}
for every $t\in \mathcal{T}\cap [t_0,+\infty)$. The continuity of the function $a(t)$ implies that inequality \eqref{fffeq11} holds everywhere in the interval $[t_0,+\infty)$.
This last inequality along with the definition of the function $a(t)$ given by formula \eqref{fffeq12}, implies that
$$\big(\log (A_1(t))\big)'\geq 2\,t^{-1}-C\,t^{-\frac{3}{2}}$$
for all $t\in [t_0,+\infty)$. Thus, integrating this inequality from $t_0$ to $t$, we conclude as in \cite[Proposition 0.4]{Col_Min_5} that
$$A_1(t)\geq t_0^{-2}\,A_1(t_0)\,e^{-\frac{2C}{\sqrt{t_0}} }\, t^2\,.$$
\end{proof}

We are now able to prove the main result of this subsection, Theorem \ref{mainTh}. First we recall that, under the assumption $H_2(M,\partial M;\Z)=0$, all regular level sets of the function $u$ are connected. This is a consequence of the maximum principle and the fact that $H_2(M,\partial M;\Z)=0$ implies that every regular level set of $u$ either it is the boundary of a compact and connected domain contained in $M\setminus\partial M$, or together with $\partial M$, it is the boundary of a compact and connected domain.

\begin{proof}[Proof of Theorem \ref{mainTh}] $\phantom{hh}$

\textcolor{red}{$(a)$} We suppose by contradiction that $A<0$.
Then, either $A_1(t)\leq 4\pi-A/t$ for all $t\in [\mathcal{C}/2, + \infty)$ or there exists some $t_0\in [\mathcal{C}/2, + \infty)$ such that $A_1(t_0)>4\pi-A/t_0$. 
We observe that this last case is not possible.
Indeed, if this last case holds, Proposition \ref{PropCol_Min} implies the existence of a positive constant $C$ such that $A_1(t)\geq C\,t^2\,$ for all $t\geq t_0$. Therefore, from the definition \eqref{defA_1} of the function $A(t)$ it follows that
\begin{equation}
 \int\limits_{\Sigma_{t}}\vert \nabla u \vert^{2}\, d\sigma\geq C \left(1+\frac{\mathcal{C}}{2t}\right)^{\!-4}\!\!\mathcal{C}^2\geq 2^{-4}C\,\mathcal{C}^2\,,
\end{equation}
but this is no possible as $\vert \nabla u\vert \to 0$ at infinity and $\int_{\Sigma_{t}}\vert \nabla u \vert \, d\sigma=4\pi\mathcal{C}$ for all $t\in [\mathcal{C}/2, + \infty)$
(this last equality is a consequence of the divergence theorem, Sard's theorem and the continuity of the function $t\mapsto \int_{\Sigma_{t}}\vert \nabla u \vert \, d\sigma$, which follows from \cite[Lemma 11 and Lemma 12]{ChodLi}). In conclusion, $A_1(t)\leq 4\pi-A/t$ for all $t\in [\mathcal{C}/2, + \infty)$. An immediate result is that the function $G(t)$, defined by \eqref{defG}, converges to zero as $t\to +\infty$. This gives $A\geq0$, as shown at the end of the proof of \cite[Proposition 4.1]{Or_1}. 
Indeed, for every $t\in[\mathcal{C}/2, + \infty)$,  the properties of the function $F(t)$ and equality \eqref{eqf50} imply
\begin{equation*}
G'(t)\geq\, \frac{\mathcal{C}^{2}}{4t^{3}}\,A\,.
\end{equation*}
Then, integrating this inequality between $\mathcal{C}/2$ and $t$, it follows
\begin{equation*}
G(t)+ A\geq \Big(-\frac{\mathcal{C}^{2}}{8t^{2}}\,+\,\frac{1}{2}\Big) A\,,
\end{equation*}
as $G(\mathcal{C}/2)=-A$.
Thus, passing to the limit for $t\to +\infty$, we get $A\geq0$, which is a contradiction.

\smallskip

To ensure clarity, let us collect some facts and consequences of this last result.
\begin{itemize}
\item[$(1)$] Arguing as above, from Proposition \ref{PropCol_Min} it follows that $A_1(t)\leq 4\pi$ for all $t\in [\mathcal{C}/2, + \infty)$. This fact implies the convergence of $G(t)$ to zero as $t\to+\infty$.
\item[$(2)$] The function $F(t)$ is continuously differentiable in the set $\mathcal{T}$ and its first derivative is given by formula \eqref{feq54}. 
Furthermore, the function $F(t)$ admits a locally absolutely continuous representative in $[\mathcal{C}/2, + \infty)$, which coincides with it in the set $\mathcal{T}$, an open subset of $[\mathcal{C}/2, +\infty)$ containing $\mathcal{C}/2$. 
Then, by Sard's theorem, the function $F(t)$ is non--decreasing on the set $\mathcal{T}$. This, along with the result $F(\mathcal{C}/2)=A\geq 0$, implies that $F$ is non--negative in $\mathcal{T}$.
\item[$(3)$] The function $G(t)$ is non--decreasing, by the last statement in $(2)$ and equality \eqref{eqf50}. Moreover, $G(\mathcal{C}/2)=-A$ and $\lim_{t\to+\infty} G(t)=0$.
\end{itemize}
Let us continue the proof of point $(a)$.
 If $A=0$, then the function $G$ is identically zero by point $(3)$. Therefore, the rest of claim $(a)$ follows from \cite[Proposition 4.2]{Or_1}.

\textcolor{red}{$(b)$} The first part of the statement is contained in point $(1)$. Now, we assume that there exists some $t_0\in[\mathcal{C}/2, + \infty)$ such that $A_1(t_0)=4\pi$. Then, $G(t_0)=0$ and it follows from point $(3)$ that the function $G$ is identically zero on the interval $[t_0,+\infty)$. A trivial consequence of this fact is that $G'(t)=0$ for all $t\in [t_0,+\infty)$. This equality, along with formula \eqref{eqf50}, implies that the function $F(t)$, defined by \eqref{feq40}, is identically zero on the set $\mathcal{T}$, by using point $(2)$. Therefore, the rest of claim $(b)$ follows from \cite[Proposition 3.2]{Or_1}.

\textcolor{red}{$(c)$} The first part of the statement follows from point $(3)$, arguing analogously as in \cite[Theorem 4.3]{Or_1}. Let us assume that there exists some $t_0\in[\mathcal{C}/2, + \infty)$ such that the equality holds in inequality \eqref{areacompwithboundary}. Since the steps are the same of those contained in Subsection \ref{Without boundary case}, for such a value, we have $G(t_0)=0$. Then, the rest of claim $(c)$ follows from the proof of the equality case in point $(b)$.

\textcolor{red}{$(d)$} It follows from point $(c)$ with $t=\mathcal{C}/2$.

\textcolor{red}{$(e)$} Proceeding as done in Subsection \ref{Without boundary case}, we have
\begin{align*}
4\pi s^{2}\,\bigg(\!1+\frac{\mathcal{C}}{2s}\bigg)^{\!4}&\leq \mathrm{Area}\left(\Biggl\{u=\frac{1-\frac{\mathcal{C}}{2s}}{1+\frac{\mathcal{C}}{2s}}\Biggr\}^{\!\!*}\right)\\
&\leq \Bigg(\,\int\limits_{\Sigma_s} \frac{1}{\vert  \nabla u\vert}\,d\sigma\Bigg)^{\frac{2}{3}} \, \Bigg(\,\int\limits_{\Sigma_s} \vert  \nabla u\vert^2\,d\sigma\Bigg)^{\frac{1}{3}}\\
&\leq \bigg[4\pi \mathcal{C}^2s^{-2}\,\bigg(\!1+\frac{\mathcal{C}}{2s}\bigg)^{\!-4}\,\bigg]^{\frac{1}{3}}\Bigg(\,\int\limits_{\Sigma_s} \frac{1}{\vert  \nabla u\vert}\,d\sigma\Bigg)^{\frac{2}{3}}\,,
\end{align*}
which implies
\begin{equation}\label{fffeq13}
\frac{\mathcal{C}}{s^2}\,\bigg(\!1+\frac{\mathcal{C}}{2s}\bigg)^{\!-2}\int\limits_{\Sigma_s} \frac{1}{\vert  \nabla u\vert}\,d\sigma\geq 4\pi s^{2}\bigg(\!1+\frac{\mathcal{C}}{2s}\bigg)^{\!6}\,.
\end{equation}
Combining the coarea formula \cite[Proposition 2.1]{BMF} with the fact that $\left\lbrace\abs{\nabla u}=0\right\rbrace$ is $\sigma$--negligible, by inequality \eqref{fffeq13} we obtain
\begin{align}
\mathrm{Vol}\left(\Biggl\{u\leq\frac{1-\frac{\mathcal{C}}{2t}}{1+\frac{\mathcal{C}}{2t}}\Biggr\}\right)
= \int\limits_{\mathcal{C}/2}^t\frac{\mathcal{C}}{s^2}\,\bigg(\!1+\frac{\mathcal{C}}{2s}\bigg)^{\!-2}\Bigg(\,\int\limits_{\Sigma_s} \frac{1}{\vert  \nabla u\vert}\,d\sigma\Bigg) \,ds \geq \int\limits_{\mathcal{C}/2}^t 4\pi s^{2}\bigg(\!1+\frac{\mathcal{C}}{2s}\bigg)^{\!6}\,ds\,.
\end{align}
We assume that there exists some $t_0\in(\mathcal{C}/2, + \infty)$ such that the equality holds in the previous inequality. Then, the equality in inequality \eqref{fffeq13} is true for almost every $s\in [\mathcal{C}/2, t_0]$.
In particular, there exists $t_1\in [\mathcal{C}/2, t_0]$ such that $G(t_1)=0$. Thus, the rest of claim $(c)$ follows from the proof of the equality case in point $(b)$.

\end{proof}

\begin{remark}
The statements of Theorem \ref{mainTh} continue to be true if we replace the assumption that $\partial M$ is minimal with the existence of some $\alpha\in\big(\!-(2\mathcal{C})^{-1},\,(2\mathcal{C})^{-1}\big]$ such that 
$ \HHH\leq \alpha\big(1-4\mathcal{C}\vert  \nabla u \vert\,\big)$ on $\partial M$, by \cite{Or_1}.
\end{remark}

\begin{ackn} The author deeply thanks Carlo Mantegazza for his thorough reading of the manuscript and for
 the many precious suggestions.
The author is a member of the INDAM--GNAMPA. 
\end{ackn}

\bibliographystyle{amsplain}
\bibliography{biblio}
 
\end{document}